\documentclass[noinfoline,11pt]{imsart}
\usepackage{amsthm,amsmath,amssymb,amsfonts,delarray}
\usepackage[dvipdfmx]{graphicx}
\usepackage[round]{natbib}
\newtheorem{thm}{Theorem}[section]
\newtheorem{lem}[thm]{Lemma}
\newtheorem{prop}[thm]{Proposition}

\theoremstyle{definition}
\newtheorem{rem}[thm]{Remark}

\begin{document}
\begin{frontmatter}
\title{On $\varepsilon$-Admissibility \\ in High Dimension and Nonparametrics \\}
\runtitle{On $\varepsilon$-Admissibility}
	\begin{aug}
\author{Keisuke Yano$\,^{1}$ \ead[label=e1]{yano@mist.i.u-tokyo.ac.jp}}
\and
\author{Fumiyasu Komaki$\,^{1,2}$ \ead[label=e2]{komaki@mist.i.u-tokyo.ac.jp}}\\
\affiliation{The University of Tokyo}
\address{$\,^{1}$Department of Mathematical Informatics,
Graduate School of
Information Science and Technology,
The University of Tokyo,
7-3-1 Hongo, Bunkyo-ku, Tokyo 113-8656, Japan}
\printead{e1,e2}\\
\address{$\,^{2}$RIKEN Brain Science Institute,
2-1 Hirosawa, Wako City,
Saitama 351-0198, Japan}
\runauthor{K. Yano and F. Komaki}
\end{aug}
\begin{abstract}
In this paper,
we discuss the use of $\varepsilon$-admissibility 
for estimation in high-dimensional and nonparametric statistical models.
The minimax rate of convergence is widely used to compare the performance of estimators in high-dimensional and nonparametric models.
However, it often works poorly as a criterion of comparison.
In such cases, the addition of comparison by $\varepsilon$-admissibility provides a better outcome.
We demonstrate the usefulness of $\varepsilon$-admissibility through high-dimensional Poisson model and Gaussian infinite sequence model,
and
present noble results.
\end{abstract}
\begin{keyword}[class=MSC]
\kwd{62C15, 62C20, 62G05}
\end{keyword}
\begin{keyword}
\kwd{Asymptotics}
\kwd{Bayes risk}
\kwd{Decision theory}
\kwd{Gaussian sequence model}
\kwd{Poisson model}
\end{keyword}
\end{frontmatter}

\section{Introduction}

Consider a statistical decision problem
in which
$\mathcal{X}$ is a sample space,
$\Theta$ is a parameter space,
and
$\mathcal{P}$ is a statistical model $\{P_{\theta}:\theta\in\Theta\}$
such that for each $\theta\in\Theta$, $P_{\theta}$ is a probability measure on $\mathcal{X}$.
Let $\mathcal{A}$ be an action space
and $\mathcal{D}$ a decision space,
comprising the whole set of measurable functions from $\mathcal{X}$ to $\mathcal{A}$.
Let $L$ be a loss function $\Theta\times \mathcal{A}\to\mathbb{R}\cup\{+\infty\}$
and $R$ a corresponding risk function defined by $R(\theta,\delta)=\int L(\theta,\delta(x)) \mathrm{d}P_{\theta}(x)$ 
for every $\theta\in\Theta$ and every $\delta\in\mathcal{D}$.

Our focus is on the use of $\varepsilon$-admissibility.
For $\varepsilon>0$, $\varepsilon$-admissibility is defined as follows:
an estimator $\delta$ is $\varepsilon$-admissible if and only if 
there exists no estimator $\tilde{\delta}$ such that for every $\theta\in\Theta$,
$R(\theta,\tilde{\delta})<R(\theta,\delta)-\varepsilon$.
In other words,
$\delta$ is
$\varepsilon$-admissible if
for any other estimator $\tilde{\delta}$,
$\delta$ is not inferior to $\tilde{\delta}$ at some $\theta\in\Theta$ when $\varepsilon$ is subtracted from the risk of $\delta$.
For $\delta\in\mathcal{D}$,
the infimum of possible values of $\varepsilon$ such that $\delta$ is $\varepsilon$-admissible is denoted by $\mathcal{R}(\Theta,\delta)$:
$$\mathcal{R}(\Theta,\delta):=\sup_{\tilde{\delta}\in\mathcal{D}}\inf_{\theta\in\Theta}[R(\theta,\delta)-R(\theta,\tilde{\delta})].$$
If $\mathcal{R}(\theta,\delta)=\varepsilon>0$, then $\delta$ is $\varepsilon$-admissible,
and
if $\delta$ is $\varepsilon$-admissible, then $\mathcal{R}(\Theta,\delta)\leq \varepsilon$.
A smaller value of $\mathcal{R}(\Theta,\delta)$ is preferable.
For further details,
see \citet{BlackwellandGirshick(1954)}, \citet{Farrell(1968)}, \citet{Ferguson_Book}, \citet{Hartigan_Book}, and \citet{HeathandSudderth(1978)}.
Although the concept of $\varepsilon$-admissibility was once widely studied in statistical decision theory,
as a research topic,
it has long been abandoned.

In the present paper,
we emphasize the use of $\varepsilon$-admissibility as a criterion for comparing the estimators in high-dimensional and nonparametric statistical models.
We show,
through two important examples,
that
by adding a comparison using the value of $\mathcal{R}(\Theta,\delta)$ to that using the minimax rate of convergence,
the performance of estimators in high-dimensional and nonparametric statistical models can be more successfully compared.
In high-dimensional and nonparametric models,
the minimax rate of convergence in an asymptotics
has been used to measure the performance of an estimator.
For example,
the minimax rate of convergence of an estimator $\delta$ in the asymptotics
in which the dimension $d$ of a parameter space $\Theta$ grows to infinity
is defined by
$d^{\alpha}$,
where $\alpha$ is the minimum number that satisfies
$0<\lim_{d\to\infty}\inf_{\delta\in\mathcal{D}}\sup_{\theta\in\Theta}R(\theta,\delta)/d^{\alpha}<\infty$.
When using the minimax approach,
the key criterion is
whether or not the rate of convergence of the estimator matches the minimax rate of convergence
(\citealp{Tsybakov_Book} and \citealp{Wasserman_Book}).
However,
the minimax rate of convergence often fails to clearly distinguish between estimators.
In such cases, adding a comparison using $\varepsilon$-admissibility can be helpful.
In the present study,
we investigated the use of $\varepsilon$-admissibility by application to two examples:
estimation of the mean in a high-dimensional Poisson model
and 
estimation of the mean in a Gaussian infinite sequence model.

We first show, using estimation of the mean in the high-dimensional Poisson model,
that
$\varepsilon$-admissibility preserves the dominating result in a finite dimensional setting, in contrast with the minimax approach.
Consider estimation of the mean in a $d$-dimensional Poisson model with an $\mathcal{L}^{1}$-constraint parameter space.
This estimation appears in discretization of an inhomogeneous Poisson point process model (see Appendix \ref{Appendix:Poissonpointprocess}).
In a setting in which $d>2$ is fixed,
it is known that the James--Stein type estimator $\hat{\theta}_{\mathrm{JS}}$
dominates the Bayes estimator $\hat{\theta}_{\mathrm{J}}$ based on Jeffreys' prior
when using the divergence loss;
see \citet{Komaki(2004)}, \citet{Komaki(2006)}, and \citet{Komaki(2015)}
and see also \citet{GhoshandYang(1988)}.
Here, $\delta$ is said to dominate $\tilde{\delta}$ if and only if $R(\theta,\delta)\leq R(\theta,\tilde{\delta})$ for all $\theta\in\Theta$
and there exists $\theta_{0}\in\Theta$ such that $R(\theta_{0},\delta)<R(\theta_{0},\tilde{\delta})$.
Unfortunately,
the minimax rate of convergence
can not determine whether $\hat{\theta}_{\mathrm{JS}}$ is superior to $\hat{\theta}_{\mathrm{J}}$
because 
the rates of convergence of both $\hat{\theta}_{\mathrm{J}}$ and a minimax estimator
are $d$.
In contrast,
by applying $\varepsilon$-admissibility,
we can decide that $\hat{\theta}_{\mathrm{J}}$ is better than $\hat{\theta}_{\mathrm{JS}}$ even in the asymptotic sense,
because a simple calculation introduced in Section \ref{Section:PoissonL1} will show
that $\lim_{d\to\infty}\mathcal{R}(\Theta,\hat{\theta}_{\mathrm{J}})/d>0$ and $\lim_{d\to\infty}\mathcal{R}(\Theta,\hat{\theta}_{\mathrm{JS}})< 2$.

We further show, through estimation of the mean in a Gaussian infinite sequence model,
that
$\varepsilon$-admissibility can quantify the degree of preference of one asymptotically minimax estimator
over another.
Consider the estimation of the mean in a Gaussian infinite sequence model with a Sobolev-type constraint parameter space.
This model is a canonical model in nonparametric statistics,
and 
has been shown to be statistically equivalent to the nonparametric regression model
(\citealp{Tsybakov_Book}, pp.~65--69).
In this context,
\citet{Zhao(2000)} demonstrated that
any Gaussian prior of which the Bayes estimator is asymptotically minimax places no mass on the parameter space.
Zhao also constructed a prior of which the Bayes estimator is asymptotically minimax and the mass on the parameter space is strictly positive.
See also \citet{ShenandWasserman(2001)}.
However, the goodness due to the strictly positive mass on the parameter space has yet been quantified.
We show that a modification of the prior discussed in \citet{Zhao(2000)} 
yields an asymptotically minimax Bayes estimator that,
from the viewpoint of $\varepsilon$-admissibility,
is superior to one based on the Gaussian prior.
This is discussed in Section \ref{Section:Gaussiansequence}.

Finally,
we address the relationship to admissibility.
Any Bayes estimator based on a prior on $\Theta$ is admissible and thus $\varepsilon$-admissible for any $\varepsilon>0$,
so that
$\varepsilon$-admissibility can not be used to compare such Bayes estimators.
However, in practice, 
estimators based on a prior that puts full mass on $\Theta$ are rarely used,
as there exist few settings in which the full information on $\Theta$ is known in advance.
The estimators discussed in Sections \ref{Section:PoissonL1} and \ref{Section:Gaussiansequence}
do not depend on knowing the full structure of $\Theta$.

The rest of the paper is organized as follows.
In Section \ref{Section:properties},
we introduce the properties of $\varepsilon$-admissibility and 
discuss its relationship with a related concept introduced by \citet{Chatterjee(2014)},
known as $C$-admissibility.
We also introduce the asymptotic notation.
Section \ref{Section:Discussions} concludes the paper.
An additional demonstration using the high-dimensional Gaussian sequence model with an $\mathcal{L}^{2}$-constraint parameter space 
is provided in Appendix \ref{Appendix:Additionalexample}.

\section{Preliminaries}\label{Section:properties}

\subsection{Bounds for $\varepsilon$-admissibility}
In this subsection,
we provide the general lower and upper bounds for $\mathcal{R}(\Theta,\delta)$
that are used in later sections.
While
these bounds are fundamental and have been widely used 
in the literature of statistical decision theoretic literature
(see, for example, Chapter 5 of \citet{LehmannandCasella_Book}),
the proofs help clarify the concept of $\varepsilon$-admissibility.
Throughout this subsection,
we use a fixed estimator $\delta\in\mathcal{D}$.

\begin{lem}
	\label{lower_maximin}
	For an estimator $\tilde{\delta}$ that dominates $\delta$,
	\begin{align*}
		\mathcal{R}(\Theta,\delta)\geq \inf_{\theta\in\Theta}[R(\theta,\delta)-R(\theta,\tilde{\delta})]\geq 0.
	\end{align*}
\end{lem}
\begin{proof}
	The first inequality follows immediately from the definition of $\mathcal{R}(\Theta,\delta)$.
	The second inequality follows since for any $\theta\in\Theta$, $R(\theta,\delta)\geq R(\theta,\tilde{\delta})$.
\end{proof}

\begin{lem}
	\label{upper_maximin}
	For a probability measure $\Pi$ on $\Theta$,
	\begin{align*}
		\mathcal{R}(\Theta,\delta)\leq \int_{\Theta}[R(\theta,\delta)-R(\theta,\delta_{\Pi})]\mathrm{d}\Pi(\theta),
	\end{align*}
	where $\delta_{\Pi}$ is the Bayes solution with respect to $\Pi$,
	i.e., the minimizer of $\int_{\Theta}R(\theta,\delta)\mathrm{d}\Pi(\theta)$.
\end{lem}
\begin{proof}
	Since for a function $f$ of $\theta$, $\inf_{\theta\in\Theta} f(\theta)\leq \int_{\Theta} f(\theta)\mathrm{d}\Pi(\theta)$,
	we have
	\begin{align*}
		\mathcal{R}(\Theta,\delta)=&\sup_{\tilde{\delta}}\inf_{\theta}[R(\theta,\delta)-R(\theta,\tilde{\delta})]
		\nonumber\\
		\leq&\sup_{\tilde{\delta}}\int_{\Theta}[R(\theta,\delta)-R(\theta,\tilde{\delta})]\mathrm{d}\Pi(\theta)
		\nonumber\\
		=&\int_{\Theta}R(\theta,\delta)\mathrm{d}\Pi(\theta)-\inf_{\tilde{\delta}}\int_{\Theta}R(\theta,\tilde{\delta})\mathrm{d}\Pi(\theta)
		\nonumber\\
		=&\int_{\Theta}R(\theta,\delta)\mathrm{d}\Pi(\theta)-\int_{\Theta}R(\theta,\delta_{\Pi})\mathrm{d}\Pi(\theta),
	\end{align*}
	where the last equality follows from the definition of the Bayes solution.
\end{proof}

Next,
we describe the relationships between $\varepsilon$-admissibility and admissibility and between $\varepsilon$-admissibility and minimaxity.
Although these relationships are not used in this paper,
they also help clarify the nature of $\varepsilon$-admissibility.
\begin{prop}
If $\delta$ is admissible, then $\mathcal{R}(\Theta,\delta)=0$.
If $\delta$ is minimax with a constant risk, then again $\mathcal{R}(\Theta,\delta)=0$.
\end{prop}
\begin{proof}
The first claim holds because from the admissibility of $\delta$, we have
$\inf_{\theta\in\Theta}[R(\theta,\delta)-R(\theta,\tilde{\delta})]\leq 0$ for any $\tilde{\delta}\in\mathcal{D}$
and thus $\sup_{\tilde{\delta}}\inf_{\theta\in\Theta}[R(\theta,\delta)-R(\theta,\tilde{\delta})]\leq 0$.
The second claim holds because 
letting $c := \inf_{\tilde{\delta}} \sup_{\theta} R( \theta , \tilde{\delta} )$ yields
\begin{align*}
	\mathcal{R}(\Theta,\delta)=&\sup_{\tilde{\delta}\in\mathcal{D}}\inf_{\theta\in\Theta}[R(\theta,\delta)-R(\theta,\tilde{\delta})]
	=c-\inf_{\tilde{\delta}}\sup_{\theta}R(\theta,\tilde{\delta})
	=0.
\end{align*}
\end{proof}

\subsection{Relationship to $C$-admissibility}

The concept of $C$-admissibility has appeared in the recent literature on estimation under the shape restriction,
and its connection to $\varepsilon$-admissibility should be noted.
An estimator $\delta$ is $C>0$-admissible if and only if
for every other estimator $\tilde{\delta}$,
there exists $\theta\in\Theta$ such that
$C\times R(\theta,\delta)\leq R(\theta,\tilde{\delta})$.
See \citet{Chatterjee(2014)} and \citet{Chenetal(2017)} for a discussion of this.

The only difference between $\varepsilon$-admissibility and $C$-admissibility is 
that $\varepsilon$-admissibility is based on the risk difference,
whereas $C$-admissibility is based on the risk ratio.
\citet{Chenetal(2017)} argues that
for a given estimator $\delta$,
the smallest value of $C$ for which $\delta$ is $C$-admissible
has a minimax interpretation:
\begin{align*}
	\sup\{C:\text{$\delta$ is $C$-admissible}\}=\inf_{\tilde{\delta}}\sup_{\theta\in\Theta}\frac{R(\theta,\tilde{\delta})}{R(\theta,\delta)}.
\end{align*}
Likewise,
the minus of the smallest value of $\varepsilon$ for which $\delta$ is $\varepsilon$-admissible
also has a minimax interpretation:
\begin{align}
-\inf\{\varepsilon:\text{$\delta$ is $\varepsilon$-admissible}\}=\inf_{\tilde{\delta}}\sup_{\theta\in\Theta}[R(\theta,\tilde{\delta})-R(\theta,\delta)].
\label{minimaxinterpretation}
\end{align}
The quantity (\ref{minimaxinterpretation}) itself is of interest.
\citet{OrlitskyandSuresh(2015)} conducted the regret analysis based on the quantity (\ref{minimaxinterpretation}) 
for some baseline estimator $\delta$.

The difference between the present paper and those of \citet{Chatterjee(2014)} and \citet{Chenetal(2017)}
is 
that
the latter addressed a universal bound for $C$
irrespective of the dimension of the parameter space and the sample size,
whereas 
our paper uses the rate of diminution of $\varepsilon$
as the dimension or the sample size grows to infinity for the performance comparison.

\subsection{Asymptotic notation}

In this subsection,
we set out the asymptotic notation used in later sections.

For positive functions $f(d)$ and $g(n)$,
the relation $f(d)\lesssim g(d)$ as $d\to\infty$ means 
that $$\lim_{d\to\infty}f(d)/g(d)<\infty.$$
The relation $f(d)\asymp g(d)$ as $d\to\infty$ means that $f(d)\lesssim g(d)$ and $g(d)\lesssim g(d)$.

\section{Poisson sequence model with $\mathcal{L}^{1}$-constraint parameter space}
\label{Section:PoissonL1}

In this section,
we present further details of the example discussed in the introduction.
Let $\mathcal{X}=\mathbb{N}^{d}$,
$\Theta=\{\theta=(\theta_{1},\ldots,\theta_{d}):\sum_{i=1}^{d}\theta_{i}/d\leq 1,\theta_{i}\geq 0, i=1,\ldots,d \}$,
and
$\mathcal{P}=\{P_{\theta}=\otimes_{i=1}^{d}\mathrm{Po}(\theta_{i}):\theta\in\Theta\}$,
where $\mathrm{Po}(\lambda)$ is a Poisson distribution with mean $\lambda$.
Let $\mathcal{A}=\mathbb{R}^{d}_{+}$ with the corresponding decision space $\mathcal{D}$.
Let $L(\theta,a)=D_{\mathrm{KL}}(P_{\theta}\mid\mid P_{a})$
with the corresponding risk function $R(\theta,\hat{\theta})$,
where $D_{\mathrm{KL}}(P_{\theta}\mid\mid P_{\theta'})$ is the Kullback--Leibler divergence from $P_{\theta}$ to $P_{\theta'}$:
\begin{align*}
D_{\mathrm{KL}}(P_{\theta}\mid\mid P_{\theta'}):=\int \log\frac{\mathrm{d}P_{\theta}}{\mathrm{d}P_{\theta'}}\mathrm{d}P_{\theta}
=\sum_{i=1}^{d}\left[\theta_{i}\log\frac{\theta_{i}}{\theta'_{i}}+\theta_{i}-\theta'_{i}\right].
\end{align*}

We discuss 
the performance of the following two estimators 
from the viewpoint of the minimax rate of convergence 
and 
that of $\varepsilon$-admissibility.
Let
\begin{align*}
	\hat{\theta}_{\mathrm{J},i}(X):=X_{i}+1/2,
	i=1,\ldots,d
\end{align*}
and
let
\begin{align*}
	\hat{\theta}_{\mathrm{JS},i}(X):=\frac{\sum_{j=1}^{d}X_{j}+1}{\sum_{j=1}^{d}X_{j}+d/2}(X_{i}+1/2),
	i=1,\ldots,d.
\end{align*}
The estimator $\hat{\theta}_{\mathrm{J}}$ is the Bayes estimator based on Jeffreys' prior
and
the estimator $\hat{\theta}_{\mathrm{JS}}$ is the James--Stein type estimator used in Poisson sequence models.
For further details, see \citet{Komaki(2004)} and \citet{Komaki(2006)}.

\subsection{Main results for the Poisson sequence model}

We first discuss the minimax rate of convergence.
The following theorem shows that, from the minimax rate of convergence,
it is impossible to determine whether $\hat{\theta}_{\mathrm{JS}}$ is better than $\hat{\theta}_{\mathrm{J}}$.
\begin{thm}\label{thm:Poissonminimax}
We have
\begin{align*}
\inf_{\hat{\theta}\in\mathcal{D}}\sup_{\theta\in\Theta}R(\theta,\hat{\theta})
\asymp
\sup_{\theta\in\Theta}R(\theta,\hat{\theta}_{\mathrm{J}})
\asymp
\sup_{\theta\in\Theta}R(\theta,\hat{\theta}_{\mathrm{JS}})
\asymp d
\end{align*}
as $d\to\infty$
\end{thm}
Next, we discuss $\varepsilon$-admissibility.
The following theorem shows that, from the viewpoint of $\varepsilon$-admissibility,
the James--Stein type estimator is superior to the Bayes estimator based on Jeffreys' prior.
\begin{thm}\label{thm:Poissonweakadmissibility}
We have
\begin{align*}
\mathcal{R}(\Theta,\hat{\theta}_{\mathrm{JS}}) \lesssim 1\lesssim d\lesssim \mathcal{R}(\Theta,\hat{\theta}_{\mathrm{J}})
\end{align*}
as $d\to\infty$.
\end{thm}
The proofs of theorems are given in the next subsection.

\subsection{Proofs of theorems}

In this subsection,
we give the proofs of Theorems \ref{thm:Poissonminimax} and \ref{thm:Poissonweakadmissibility}

\begin{proof}[Proof of Theorem \ref{thm:Poissonminimax}]
This proof relies on the fact that a minimax risk is bounded below by a Bayes risk:
For a probability distribution $\Pi$ on $\Theta$,
we have
\begin{align}
	\inf_{\hat{\theta}\in\mathcal{D}}\sup_{\theta\in\Theta}R(\theta,\hat{\theta})
	\geq\int R(\theta,\hat{\theta}_{\Pi}) \mathrm{d}\Pi(\theta),
\label{eq:Bayesriskbound}
\end{align}
where $\hat{\theta}_{\Pi}$ is the Bayes solution with respect to $\Pi$.
Let $\Pi$ be
\begin{align*}
	\Pi(\mathrm{d}\theta)=\frac{1}{2}\delta_{0}(\mathrm{d}\theta)+
	\frac{1}{2}\delta_{d}(\mathrm{d}\|\theta\|_{1})\otimes
	\mathrm{Dir}\left(\frac{1}{2},\cdots,\frac{1}{2}\right)\left(\mathrm{d}\frac{\theta_{1}}{\|\theta\|_{1}},\ldots,\mathrm{d}\frac{\theta_{d}}{\|\theta\|_{1}}\right),
\end{align*}
where $\delta_{x}$ is the Dirac measure having a mass on $x$,
$\mathrm{Dir} (1/2,\ldots,1/2) \allowbreak ( \mathrm{d}x_{1}, \ldots, \mathrm{d}x_{d} )$ 
is the Dirichlet distribution of which the density is proportional to $ x_{1}^{1/2-1}\times \cdots \times x_{d}^{1/2-1}$,
and
$\| \theta \|_{1} := \sum_{i} | \theta_{i} |$.

First,
we show that the Bayes risk $b(\Pi):=\int R(\theta,\hat{\theta}_{\Pi})\mathrm{d}\Pi(\theta)$ is of order $d$
by assuming that the following two claims hold:
\begin{enumerate}
\setlength{\leftskip}{2cm}
\item[Claim C1.] $b(\Pi)$ is given by
\begin{align}
	b(\Pi)=\frac{d}{2}&\left\{\mathrm{e}^{-d}\log(1+\mathrm{e}^{d})+\psi(3/2)-\psi(d/2+1)
	\right.
	\nonumber\\
	&\left.+\mathrm{E}_{X\sim\mathrm{Po}(d)}\log(d/2+X)
	\right.
	\nonumber\\
	&\left.-\mathrm{E}_{b\sim\mathrm{Beta}(3/2,d/2-1)}\mathrm{E}_{X\sim\mathrm{Po}(bd)}\log(1/2+X)\right\},
	\label{eq:Bayesrisk}
\end{align}
where $\psi(\cdot)$ is the digamma function, that is, the derivative of the log of the Gamma function;
\item[Claim C2.] for any $\varepsilon\in(0,1)$, the asymptotic inequality
	\begin{align*}\mathrm{E}_{X\sim\mathrm{Po}(d)}\log(d/2+X)\geq\log(3d/2)+\log(1-\varepsilon)+\mathrm{o}_{\varepsilon}(1)\end{align*} holds,
\end{enumerate}
where $\mathrm{o}_{\varepsilon}(1)$ is the $\mathrm{o}(1)$ term depending on $\varepsilon$.

Applying Jensen's inequality to $x\to\log(1/2+x)$ yields
\begin{align*}
\mathrm{E}_{b\sim\mathrm{Beta}(3/2,d/2-1)}\mathrm{E}_{X\sim\mathrm{Po}(bd)}\log(1/2+X)\leq 
\mathrm{E}_{b\sim\mathrm{Beta}(3/2,d/2-1)}\log(1/2+bd).
\end{align*}
Since $\mathrm{Beta}(3/2,d/2-1)$ converges weakly to $\delta_{1}$ as $d \to \infty$,
we have
\begin{align}
	\mathrm{E}_{b\sim\mathrm{Beta}(3/2,d/2-1)}\log(1/2+bd)=\log d+\mathrm{o}(1).
	\label{eq:lastterm}
\end{align}
Note that $b\in[0,1]\to\log(1/2+b)$ is bounded and continuous.
Thus,
from Claims C1 and C2,
from the asymptotic inequality (\ref{eq:lastterm}),
and
from the asymptotic relationship that $\psi(d/2+1)=\log(d/2)+\mathrm{o}(1)$ (see \citet{NIST}),
we have
\begin{align*}
	b(\Pi)\geq\frac{d}{2}(1+\psi(3/2)+\log(1-\varepsilon)+\mathrm{o}_{\varepsilon}(1)).
\end{align*}
Taking $\varepsilon$ such that $1+\psi(3/2)+\log(1-\varepsilon)>0$,
it is shown that the Bayes risk $b(\Pi):=\int R(\theta,\hat{\theta}_{\Pi})\mathrm{d}\Pi(\theta)$ is of order $d$.

\vspace{5mm}

Next, we prove that Claim C1 holds.
Let $\mathrm{S}:=\mathrm{Dir}(1/2,\ldots,1/2)$.
For $i=1,\ldots,d$,
we have
\begin{align*}
	\hat{\theta}_{\Pi,i}(X)=
	\begin{cases}
		d\times \frac{\int w_{i}(\prod_{j=1}^{d}w_{j}^{x_{j}})\mathrm{d}\mathrm{S}(w_{1},\ldots,w_{d}) }
		{ \int (\prod_{j=1}^{d}w_{j}^{x_{j}})\mathrm{d}\mathrm{S}(w_{1},\ldots,w_{d}) }	
		& \text{ if $x_{j}\neq 0$ for some $j$}, \\
		d\times \frac{\mathrm{e}^{-d}}{1+\mathrm{e}^{-d}}\int w_{i}\mathrm{d}\mathrm{S}(w_{1},\ldots,w_{d})
		& \text{ if all $x_{j}$s are 0}.
	\end{cases}
\end{align*}
Substituting the above expression of $\hat{\theta}_{\Pi}$ into $R(\theta,\hat{\theta}_{\Pi})$,
we have
\begin{align}
	R(\theta,\hat{\theta}_{\Pi})=&
	\sum_{i=1}^{d}\theta_{i}\log\frac{\theta_{i}}{d/d}+\left(\sum_{i=1}^{d}\theta_{i}\right)\mathrm{e}^{-\sum_{i=1}^{d}\theta_{i}}\log(1+\mathrm{e}^{d})
	\nonumber\\
	&-\left(\sum_{i=1}^{d}\theta_{i}\right)+d-d\mathrm{e}^{-\sum_{i=1}^{d}\theta_{i}}\frac{1}{1+\mathrm{e}^{-d}}
	\nonumber\\
	&-\left(\sum_{i=1}^{d}\theta_{i}\right)\log d
	\nonumber\\
	&+\left(\sum_{i=1}^{d}\theta_{i}\right)\mathrm{E}_{X\sim\mathrm{Po}(\sum_{i=1}^{d}\theta_{i})}\log(d/2+X)
	\nonumber\\
	&-\sum_{i=1}^{d}\theta_{i}\mathrm{E}_{X_{i}\sim\mathrm{Po}(\theta_{i})}\log(1/2+X_{i}).
	\label{eq:risk_Pi}
\end{align}
Since 
\begin{align*}
	\int\sum_{i=1}^{d}\theta_{i}\log\theta_{i}\mathrm{d}\Pi(\theta)=&d\log d+d\int \sum_{i}^{d}w_{i}\log w_{i}\mathrm{d}S(w_{1},\ldots,w_{d})
	\nonumber\\
	=&d\log d+\psi(3/2)-\psi(d/2+1),
\end{align*}
taking the expectation of the right hand side of (\ref{eq:risk_Pi}) 
over $\theta$ with respect to $\Pi$ 
shows that Claim C1 holds.

\vspace{5mm}

Finally, we prove that Claim C2 holds.
We have
\begin{align*}
	\mathrm{E}_{X\sim\mathrm{Po}(d)}\log(d/2+X)&=\log(3d/2)+\mathrm{E}_{X\sim\mathrm{Po}(d)}\log\left(1+\frac{2}{3\sqrt{d}}\frac{X-d}{\sqrt{d}}\right).
\end{align*}
Since for a random variable $X$ distributed according to $\mathrm{Po}(d)$,
$(X-d)/d$ converges to $0$ in probability as $d \to \infty$,
we have,
for any $\varepsilon\in(0,1)$,
\begin{align*}
	\mathrm{E}_{X\sim\mathrm{Po}(d)}&\log\left(1+\frac{2}{3\sqrt{d}}\frac{X-d}{\sqrt{d}}\right)
	\nonumber\\
	&\geq\log(1-\varepsilon)+\mathrm{E}_{X\sim\mathrm{Po}(d)}1_{|2(X-d)/d|>\varepsilon}\log\left(1+\frac{2}{3\sqrt{d}}\frac{X-d}{\sqrt{d}}\right)
	\nonumber\\
	&\geq\log(1-\varepsilon)+\mathrm{Pr}(|2(X-d)/d|>\varepsilon) \log(1/3)
	\nonumber\\
	&=\log(1-\varepsilon)+\mathrm{o}_{\varepsilon}(1),
\end{align*}
where $X$ is a random variable distributed according to $\mathrm{Po}(d)$.
This completes the proof.

\end{proof}

\begin{proof}[Proof of Theorem \ref{thm:Poissonweakadmissibility}]
It suffices to show that for any $d\in\mathbb{N}$,
\begin{align}
	\mathcal{R}(\Theta,\hat{\theta}_{\mathrm{J}})
	\geq - d\log\left(1+(d/2-1)\frac{1-\mathrm{e}^{-d}}{d}\right)
	+d/2-1
	\label{eq:risk_Jeffreys}
\end{align}
and
that
for any $d\in\mathbb{N}$,
\begin{align}
	\mathcal{R}(\Theta,\hat{\theta}_{\mathrm{JS}})
	\leq \frac{1}{2}.
	\label{eq:risk_JS}
\end{align}

\vspace{5mm}
The proof of (\ref{eq:risk_Jeffreys}) follows the proof that $\hat{\theta}_{\mathrm{JS}}$ dominates $\hat{\theta}_{\mathrm{J}}$;
see \citet{Komaki(2004)}.
Applying Lemma \ref{upper_maximin} with $\tilde{\delta}=\hat{\theta}_{\mathrm{JS}}$ yields
\begin{align*}
	\mathcal{R}(\Theta,\hat{\theta}_{\mathrm{J}})
	\geq
	\inf_{\theta\in\Theta}[R(\theta,\hat{\theta}_{\mathrm{J}})-R(\theta,\hat{\theta}_{\mathrm{JS}})].
\end{align*}
Since 
\begin{align*}
	R(\theta,\hat{\theta}_{\mathrm{J}})-R(\theta,\hat{\theta}_{\mathrm{JS}})
	=\mathrm{E}_{X\sim P_{\theta}}\sum_{i=1}^{d}\left[\theta_{i}\log\frac{\hat{\theta}_{\mathrm{JS},i}(X) }{\hat{\theta}_{\mathrm{J},i}(X)} 
	+\hat{\theta}_{\mathrm{J},i}(X) -\hat{\theta}_{\mathrm{JS},i}(X) \right],
\end{align*}
we have
\begin{align*}
	\mathcal{R}(\Theta,\hat{\theta}_{\mathrm{J}})
	\geq
	\inf_{\theta\in\Theta}
	\mathrm{E}_{X\sim P_{\theta}}
	\left[
		\sum_{i=1}^{d}\theta_{i}\log\frac{\sum_{j=1}^{d}X_{j}+1}{\sum_{j=1}^{d}X_{j}+d/2}+(d/2-1)
	\right].
\end{align*}
Since the distribution of $Z=\sum_{i=1}^{d}X_{i}$ is $\mathrm{Po}(\mu)$ with $\mu=\sum_{i=1}^{d}\theta_{i}$,
we have
\begin{align*}
\inf_{\theta\in\Theta}
&\mathrm{E}_{X\sim P_{\theta}}
	\left[
		\sum_{i=1}^{d}\theta_{i}\log\frac{\sum_{j=1}^{d}X_{j}+1}{\sum_{j=1}^{d}X_{j}+d/2}+(d/2-1)
	\right]
\nonumber\\
	&=
	\inf_{\mu\in[0,d]}
	\mathrm{E}_{Z\sim\mathrm{Po}(\mu)}
	\left[
		-\mu\log\left(1+\frac{d/2-1}{Z+1} \right)+ (d/2-1)
	\right]
	\nonumber\\
	&\geq
	\inf_{\mu\in[0,d]}
	\left\{-\mu \log\left(1+ \mathrm{E}_{Z\sim \mathrm{Po}(\mu)}\frac{d/2-1}{Z+1}\right)\right\} +(d/2-1)
	\nonumber\\
	&=
	\inf_{\mu\in[0,d]}
	\left\{-\mu \log\left(1+ (d/2-1)\frac{1-\mathrm{e}^{-\mu}}{\mu}\right)\right\} +(d/2-1)
	\nonumber\\
	&= -d\log\left(1+d/2-1\right) + (d/2-1).
\end{align*}
Here the first inequality follows from Jensen's inequality
and the third equality from the identity
\begin{align*}
	\mathrm{E}_{Z\sim\mathrm{Po}(\mu)}\left[1/(Z+1)\right]=\{1-\mathrm{e}^{-\mu}\}/\mu.
\end{align*}

\vspace{5mm}
The proof of (\ref{eq:risk_JS}) immediately follows from Lemma \ref{upper_maximin} with $\Pi=\delta_{0}$,
which completes the proof of Theorem \ref{thm:Poissonweakadmissibility}.
\end{proof}

\section{Gaussian infinite sequence model with Sobolev-type constraint parameter space}
\label{Section:Gaussiansequence}

In this section, we consider estimation of the mean in a Gaussian infinite sequence model with Sobolev-type constraint parameter space.
Let $\mathcal{X}=\mathbb{R}^{\infty}$.
Let $\Theta=\{\theta=(\theta_{1},\theta_{2}\ldots,)\in l^{2}:\sum_{i=1}^{\infty}i^{2\alpha}\theta_{i}^{2}\leq B\}$
with $\alpha>0$ and $B>0$
and
$\mathcal{P}=\{P_{\theta}=\otimes_{i=1}^{\infty}\mathcal{N}(\theta_{i},1/n):\theta\in\Theta\}$.
The hyperparameter $\alpha$ controls the smoothness level of a true function in the nonparametric regression model
and
the hyperparameter $B$ controls the volume of the parameter space.
In this paper, we assume that $\alpha$ is known.
Even in the setting in which $\alpha$ is known,
the results in this section are noble.
Let $\mathcal{A}=\mathbb{R}^{\infty}$ with the corresponding decision space $\mathcal{D}$
and
$L(\theta,a)=\|\theta-a\|^{2}$ with the corresponding risk function $R(\theta,\hat{\theta})$,
where $\|b\|^{2}:=\sum_{i=1}^{\infty}b_{i}^{2}$ for $b\in\mathbb{R}^{\infty}$.
In this section, we consider the asymptotics in which the sample size $n$ grows to $\infty$.

Let $\hat{\theta}_{\mathrm{G}}$ be the Bayes estimator based on the Gaussian prior
$$\mathrm{G}=\otimes_{i=1}^{\infty}\mathcal{N}(0,i^{-2\alpha-1})$$
and
$\hat{\theta}_{\mathrm{S}}$ be the Bayes estimator based on the prior
$$\mathrm{S}=\sum_{d=1}^{\infty}M(d)\left[\left\{\otimes_{i=1}^{d}\mathcal{N}(0,d^{2\alpha+1}i^{-2\alpha-1}/n)\right\}
\otimes\left\{\otimes_{i=d+1}^{\infty}\mathcal{N}(0,0)\right\}\right],$$
where $M(d)=\mathrm{e}^{-ad}/\sum_{i=1}^{\infty}\mathrm{e}^{-ai}$.

\begin{rem}
The prior $\mathrm{S}$ is discussed in \citet{YanoandKomaki(2017)},
and
is a modification of the compound prior in \citet{Zhao(2000)}.
The compound prior $\mathrm{C}$ is given as follows:
$$\mathrm{C}=\sum_{d=1}^{\infty}M(d)\left[\left\{\otimes_{i=1}^{d}\mathcal{N}(0,i^{-2\alpha-1})\right\}
\otimes\left\{\otimes_{i=d+1}^{\infty}\mathcal{N}(0,0)\right\}\right].$$
The modification is necessary to ensure that $\mathcal{R}(\Theta,\hat{\theta}_{\mathrm{S}})$ remains sufficiently small.
Roughly speaking, it does require the prior mass condition under which the prior puts nearly the full mass on $\Theta$
to make $\mathcal{R}(\Theta,\hat{\theta}_{\mathrm{S}})$ sufficiently small;
for further details,
see the proof of Theorem \ref{thm:Sobolevweakadmissibility_S} below.
The mass placed on $\Theta$ by the compound prior $\mathrm{C}$ is strictly less than 1 even as $n \to \infty$,
whereas
that by the prior $\mathrm{S}$ grows to 1 as $n\to \infty$ for a fixed $B>0$;
see Lemma \ref{PriorMeasureonEllipsoid}.
To demonstrate that
the compound prior $\mathrm{C}$ places a mass on $\Theta$ that is strictly less than 1 even as $n \to \infty$,
we have
$\mathrm{C}\left(\Theta\right)\leq \mathrm{Pr}(N^{2}\leq B)<1$,
where
$N$ is a one-dimensional standard normal random variable.
\end{rem}

\subsection{Existing result for the Gaussian sequence model}

The following existing result shows that, from the viewpoint of the minimax rate of convergence,
$\hat{\theta}_{\mathrm{G}}$ and $\hat{\theta}_{\mathrm{S}}$ yield the same performance.
\begin{lem}[Theorem 5.1.~in \citet{Zhao(2000)} and Theorem 2 in \citet{YanoandKomaki(2017)}]
For any $\alpha>0$ and any $B>0$,
we have
\begin{align*}
\inf_{\hat{\theta}\in\mathcal{D}}\sup_{\theta\in\Theta}R(\theta,\hat{\theta})
\asymp
\sup_{\theta\in\Theta}R(\theta,\hat{\theta}_{\mathrm{G}})
\asymp
\sup_{\theta\in\Theta}R(\theta,\hat{\theta}_{\mathrm{S}})
\asymp
\left(1/n\right)^{2\alpha/(2\alpha+1)}
\end{align*}
as $n\to \infty$.
\end{lem}

\subsection{Main results for the Gaussian sequence model}

The noble results presented in this subsection 
show that, from the viewpoint of $\varepsilon$-admissibility,
$\hat{\theta}_{\mathrm{S}}$ is superior to $\hat{\theta}_{\mathrm{G}}$ 
in the case that $\alpha=1$ and $B=1$ or the case that $B$ is sufficiently small.
Numerical evaluations also show the superiority of $\hat{\theta}_{\mathrm{S}}$ over $\hat{\theta}_{\mathrm{G}}$ 
for any $\alpha>0$ and for any $B>0$.
The results are based on two theorems:
Theorem \ref{thm:Sobolevweakadmissibility_G}
shows that
$\mathcal{R}(\Theta,\hat{\theta}_{\mathrm{G}})$ is 
of the same order as the maximum risk of the estimator,
which indicates that there exists an estimator $\hat{\theta}$ 
such that 
$$R(\theta,\hat{\theta})+\mathrm{O}(n^{-2\alpha/(2\alpha+1)})<R(\theta,\hat{\theta}_{\mathrm{G}})<\mathrm{O}(n^{-2\alpha/(2\alpha+1)})$$
for all $\theta\in\Theta$.
Theorem \ref{thm:Sobolevweakadmissibility_S} shows that
$\mathcal{R}(\Theta,\hat{\theta}_{\mathrm{S}})$ has an exponential decay.

\begin{thm}\label{thm:Sobolevweakadmissibility_G}
There exists a constant $c$ depending only on $B$ and $\alpha$  such that
the inequality
\begin{align*}
	\lim_{n\to\infty}\mathcal{R}(\Theta,\hat{\theta}_{\mathrm{G}})/n^{-2\alpha/(2\alpha+1)}>c
\end{align*}
holds.
For $B=1$ and $\alpha=1$, $c$ is taken to be strictly positive.
For any $\alpha>0$, $c$ is taken to be strictly positive if $B$ is sufficiently small.
\end{thm}
\begin{thm}\label{thm:Sobolevweakadmissibility_S}
We have
\begin{align*}
	\mathcal{R}(\Theta,\hat{\theta}_{\mathrm{S}}) 
	\lesssim
	\exp\left\{-\frac{a}{2}\left(nB\right)^{\frac{1}{4\alpha+2}}\right\}
\end{align*}
as $n\to \infty$.
\end{thm}
Proofs are provided in Subsection \ref{subsec:proof_Gaussian}.

\begin{figure}[htb]
	\begin{minipage}{0.4\hsize}
	\begin{center}
		\includegraphics[width=0.90\hsize]{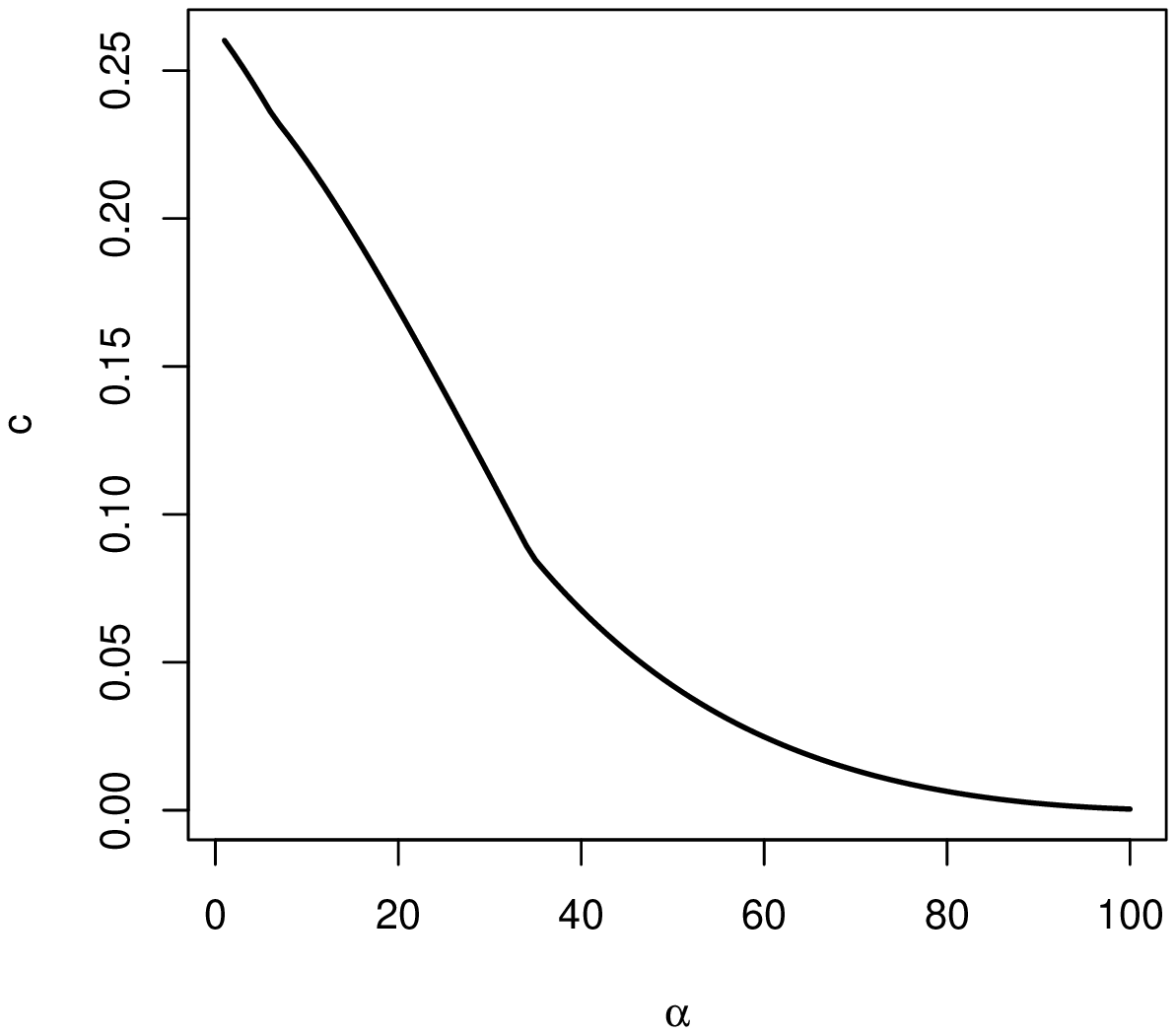}
	\caption{Possible choice of $c$ in Theorem \ref{thm:Sobolevweakadmissibility_G} with $B=1$}
	\label{Fig:alphaval}
	\end{center}
	\end{minipage}
	\begin{minipage}{0.4\hsize}
	\begin{center}
		\includegraphics[width=0.90\hsize]{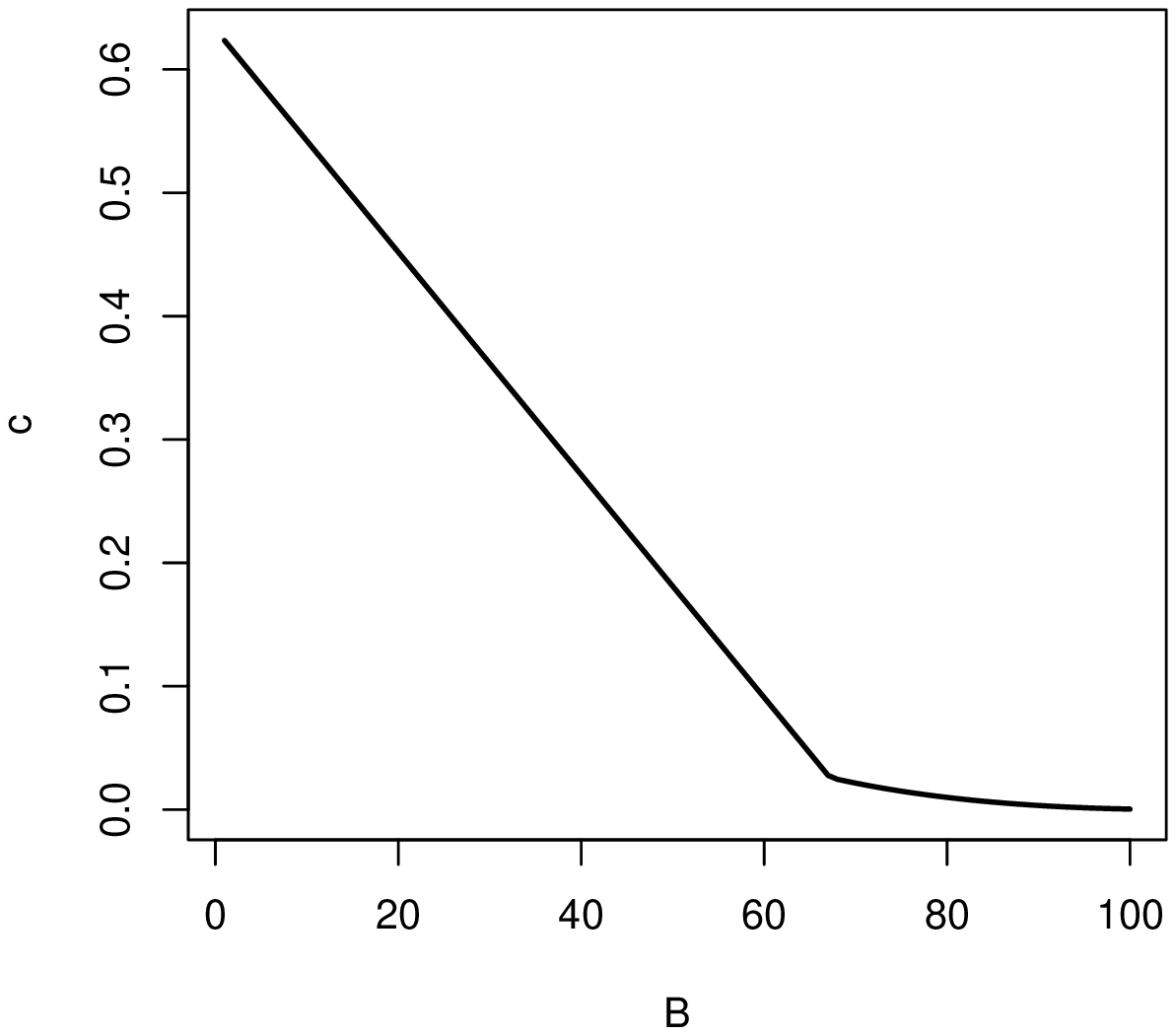}
	\caption{Possible choice of $c$ in Theorem \ref{thm:Sobolevweakadmissibility_G} with $\alpha=1$}
	\label{Fig:Bval}
	\end{center}
	\end{minipage}
\end{figure}

Although we do not provide a proof of the strict positivity of $c$ in Theorem \ref{thm:Sobolevweakadmissibility_G} for general settings,
numerical evaluations (see Figures \ref{Fig:alphaval} and \ref{Fig:Bval})
show that we can assume strict positivity of $c$ in the case that $\alpha=1$ or in the case that $B=1$.
Here, the choice of $c$ in Figures \ref{Fig:alphaval} and \ref{Fig:Bval} is described at the beginning of the proof of Theorem \ref{thm:Sobolevweakadmissibility_G}.

\subsection{Proofs of theorems}\label{subsec:proof_Gaussian}

\begin{proof}[Proof of Theorem \ref{thm:Sobolevweakadmissibility_G}]
Let $\hat{\theta}_{\mathrm{G}(s)}$ be the Bayes estimator based on $\otimes_{i=1}^{\infty}\mathcal{N}(0,si^{-2\alpha-1})$.
For $s\in(0,1)$,
we will show that
\begin{align}
\inf_{\theta\in\Theta}[R&(\theta,\hat{\theta}_{\mathrm{G}})-R(\theta,\hat{\theta}_{\mathrm{G}(s)})]
\nonumber\\
	=&n^{-\frac{2\alpha}{2\alpha+1}}
	\times\left[ B\left(\frac{2\alpha+2}{4\alpha+1}\right)^{\frac{2\alpha+2}{2\alpha+1}}
	\left\{\frac{1}{\{1+\frac{2\alpha+2}{4\alpha+1}\}^{2}}-\frac{s^{-2}}{\{1+\frac{2\alpha+2}{s(4\alpha+1)}\}^{2}}\right\}
	\right.
	\nonumber\\
	&\left.\quad\quad\quad\quad+(1-s^{\frac{1}{2\alpha+1}})\int_{0}^{\infty}\frac{1}{(1+x^{2\alpha+1})^{2}}\mathrm{d}x-n^{-\frac{1}{2\alpha+1}}\right].
\end{align}
Letting $c$ be the supremum of the right hand side in the above inequality,
will completes the proof.
The strict positivity of $c$ for the specific settings is proved in the last step.

\vspace{5mm}
First,
a direct evaluation of the risks yields
\begin{align*}
R&(\theta,\hat{\theta}_{\mathrm{G}})-R(\theta,\hat{\theta}_{\mathrm{G}(s)})
\nonumber\\
=&\sum_{i=1}^{\infty}\theta_{i}^{2}\left\{\left(\frac{i^{2\alpha+1}/n}{1+i^{2\alpha+1}/n}\right)^{2}
-\left(\frac{i^{2\alpha+1}/\{ns\}}{1+i^{2\alpha+1}/\{ns\}}\right)^{2}\right\}
	\nonumber\\
	&+n^{-1}\sum_{i=1}^{\infty}\left\{\left(\frac{1}{1+i^{2\alpha+1}/n}\right)^{2}-\left(\frac{1}{1+i^{2\alpha+1}/\{ns\}}\right)^{2}\right\}
\nonumber\\
=&\left(\sum_{j=1}^{\infty}j^{2\alpha}\theta_{j}^{2}\right)\sum_{i=1}^{\infty}
	\frac{i^{2\alpha}\theta_{i}^{2}}{\sum_{j=1}^{\infty}j^{2\alpha}\theta_{j}^{2}}
	\left\{\left(\frac{i^{2\alpha+1}/n}{1+i^{2\alpha+1}/n}\right)^{2}
	-\left(\frac{i^{2\alpha+1}/\{ns\}}{1+i^{2\alpha+1}/\{ns\}}\right)^{2}\right\}
	\nonumber\\
	&+n^{-1}\sum_{i=1}^{\infty}\left\{\left(\frac{1}{1+i^{2\alpha+1}/n}\right)^{2}-\left(\frac{1}{1+i^{2\alpha+1}/\{ns\}}\right)^{2}\right\}.
\end{align*}
For $\theta=0$,
\begin{align*}
	R(\theta,\hat{\theta}_{\mathrm{G}})-R(\theta,\hat{\theta}_{\mathrm{G}(s)})
	\geq
	n^{-1}\sum_{i=1}^{\infty}\left\{\left(\frac{1}{1+i^{2\alpha+1}/n}\right)^{2}-\left(\frac{1}{1+i^{2\alpha+1}/\{ns\}}\right)^{2}\right\}.
\end{align*}
For $\theta\neq 0$,
a calculation in Appendix \ref{Appendix:details_firstterm} yields
\begin{align*}
R&(\theta,\hat{\theta}_{\mathrm{G}})-R(\theta,\hat{\theta}_{\mathrm{G}(s)})
\nonumber\\
\geq&B\left(\frac{2\alpha+2}{4\alpha+1}\right)^{\frac{2\alpha+2}{2\alpha+1}}n^{-\frac{2\alpha}{2\alpha+1}}
	\left[
		\frac{1}{\{1+\frac{2\alpha+2}{4\alpha+1}\}^{2}}-\frac{s^{-2}}{\{1+\frac{2\alpha+2}{s(4\alpha+1)}\}^{2}}
	\right]
	\nonumber\\
	&+n^{-1}\sum_{i=1}^{\infty}\left\{\left(\frac{1}{1+i^{2\alpha+1}/n}\right)^{2}-\left(\frac{1}{1+i^{2\alpha+1}/\{ns\}}\right)^{2}\right\}.
\end{align*}
Therefore,
we have
\begin{align}
	\inf_{\theta\in\Theta}[R&(\theta,\hat{\theta}_{\mathrm{G}})-R(\theta,\hat{\theta}_{\mathrm{G}(s)})]
\nonumber\\
=&B\left(\frac{2\alpha+2}{4\alpha+1}\right)^{\frac{2\alpha+2}{2\alpha+1}}n^{-\frac{2\alpha}{2\alpha+1}}
	\left[
		\frac{1}{\{1+\frac{2\alpha+2}{4\alpha+1}\}^{2}}-\frac{s^{-2}}{\{1+\frac{2\alpha+2}{s(4\alpha+1)}\}^{2}}
	\right]
	\nonumber\\
	&+n^{-1}\sum_{i=1}^{\infty}\left\{\left(\frac{1}{1+i^{2\alpha+1}/n}\right)^{2}-\left(\frac{1}{1+i^{2\alpha+1}/\{ns\}}\right)^{2}\right\}.
	\label{eq:riskexp}
\end{align}
By convergence of the Riemann sum $\sum_{i=1}^{\infty}\{1+(i/N)^{2\alpha+1}\}^{-2}$ for a positive number $N$,
we have
\begin{align}
	\sum_{i=1}^{\infty}\frac{1}{(1+(i/N)^{2\alpha+1})^{2}}
	\leq
	N\int_{0}^{\infty}\frac{1}{(1+x^{2\alpha+1})^{2}}\mathrm{d}x
	\leq
	\sum_{i=1}^{\infty}\frac{1}{(1+(i/N)^{2\alpha+1})^{2}}+1.
	\label{eq:Riemannsum}
\end{align}
Combining (\ref{eq:Riemannsum}) with (\ref{eq:riskexp}) yields
\begin{align}
\inf_{\theta\in\Theta}[R&(\theta,\hat{\theta}_{\mathrm{G}})-R(\theta,\hat{\theta}_{\mathrm{G}(s)})]
\nonumber\\
\geq&B\left(\frac{2\alpha+2}{4\alpha+1}\right)^{\frac{2\alpha+2}{2\alpha+1}}n^{-\frac{2\alpha}{2\alpha+1}}
	\left[
		\frac{1}{\{1+\frac{2\alpha+2}{4\alpha+1}\}^{2}}-\frac{s^{-2}}{\{1+\frac{2\alpha+2}{s(4\alpha+1)}\}^{2}}
	\right]
	\nonumber\\
&+n^{-\frac{2\alpha}{2\alpha+1}}\{1-s^{\frac{1}{2\alpha+1}}\}\int_{0}^{\infty}\frac{1}{(1+x^{2\alpha+1})^{2}}\mathrm{d}x-n^{-1}
\nonumber\\
	=&n^{-\frac{2\alpha}{2\alpha+1}}
	\times\left[ B\left(\frac{2\alpha+2}{4\alpha+1}\right)^{\frac{2\alpha+2}{2\alpha+1}}
	\left\{\frac{1}{\{1+\frac{2\alpha+2}{4\alpha+1}\}^{2}}-\frac{s^{-2}}{\{1+\frac{2\alpha+2}{s(4\alpha+1)}\}^{2}}\right\}
	\right.
	\nonumber\\
	&\left.\quad\quad\quad+(1-s^{\frac{1}{2\alpha+1}})\int_{0}^{\infty}\frac{1}{(1+x^{2\alpha+1})^{2}}\mathrm{d}x-n^{-\frac{1}{2\alpha+1}}\right].
\end{align}
For $B=1$ and $\alpha=1$, taking $s=0.9$ confirms that the right hand side is positive.
The choice of $s$ follows from the direct evaluation of the integral $\int_{0}^{\infty} (1+x^{2\alpha+1})^{-2} \mathrm{d}x$.
For an arbitrary $s\in(0,1)$ and sufficiently small $B>0$, 
the right hand side is positive.
This completes the proof.

\end{proof}

\vspace{5mm}

\begin{proof}[Proof of Theorem \ref{thm:Sobolevweakadmissibility_S}]
	Let $T=\lfloor (nB)^{1/(4\alpha+2)} \rfloor$.
	Let $\widetilde{\mathrm{S}}$ be the probability distribution obtained by restricting $\mathrm{S}$ to $\Theta$.
	The corresponding Bayes estimator is denoted by $\hat{\theta}_{\widetilde{\mathrm{S}}}$.
	Let $\mathcal{D}^{*}:=\{\delta\in (l_{2})^{\mathbb{R}^{\infty}}:\delta(x)\in\Theta\}$

	From Lemma \ref{upper_maximin}, it suffices to show that
	\begin{align}
		\inf_{\delta}\int R(\theta,\delta)\mathrm{d}\widetilde{\mathrm{S}}(\theta)
		+\mathrm{O}(\exp(-aT))
		\geq
		\int R(\theta,\hat{\theta}_{\mathrm{S}})\mathrm{d}\widetilde{\mathrm{S}}(\theta).
	\label{eq:gammaBayes}
	\end{align}
	Since $\hat{\theta}_{\widetilde{\mathrm{S}}}$ is included in $\mathcal{D}^{*}$,
	\begin{align*}
	\inf_{\delta}\int R(\theta,\delta)\mathrm{d}\widetilde{\mathrm{S}}(\theta)
	=
	\inf_{\delta\in\mathcal{D}}\int R(\theta,\delta)\mathrm{d}\widetilde{\mathrm{S}}(\theta).
	\end{align*}
	Since $1_{\theta\in\Theta}=1_{\theta\in l_{2}}-1_{\theta\in\Theta^{\mathrm{c}}}$,
	we have
	\begin{align*}
		\inf_{\delta\in\mathcal{D}^{*}}&\int R(\theta,\delta)\mathrm{d}\widetilde{\mathrm{S}}(\theta)
	\nonumber\\
		&\geq
	\frac{1}{\mathrm{S}(\Theta)}
	\left[
		\inf_{\delta\in\mathcal{D}^{*}} \int R(\theta,\delta)\mathrm{d}\mathrm{S}(\theta)
		-\sup_{\delta\in\mathcal{D}^{*}} \int_{\theta\in\Theta^{\mathrm{c}}} R(\theta,\delta)\mathrm{d}\mathrm{S}(\theta)
	\right]
\nonumber\\
&\geq
\frac{1}{\mathrm{S}(\Theta)}
\left[
	\inf_{\delta} \int R(\theta,\delta)\mathrm{d}\mathrm{S}(\theta)
	-\sup_{\delta\in\mathcal{D}^{*}} \int_{\theta\in\Theta^{\mathrm{c}}} R(\theta,\delta)\mathrm{d}\mathrm{S}(\theta)
\right].
	\end{align*}
	Since $||\theta-\delta(X)||^{2} \leq 2(B+||\theta||^{2})$ for $\delta\in\mathcal{D}^{*}$
	and
	$\inf_{\delta}\int R(\theta,\delta)\mathrm{d}\mathrm{S}=\int R(\theta,\hat{\theta}_{\mathrm{S}}) \mathrm{d}\mathrm{S}(\theta)$,
	we have
	\begin{align*}
	&\left[
	\inf_{\delta} \int R(\theta,\delta)\mathrm{d}\Lambda(\theta)
	-\sup_{\delta\in\mathcal{D}^{*}} \int_{\theta\in\mathcal{E}^{\mathrm{c}}} R(\theta,\delta)\mathrm{d}\mathrm{S}(\theta)
\right]\nonumber\\
	&\geq
\left[\int 1_{\theta\in\Theta} R(\theta,\hat{\theta}_{\mathrm{S}})\mathrm{d}\mathrm{S}(\theta)
	-(2B\mathrm{S}(\Theta^{\mathrm{c}})+2\sqrt{\mathrm{E}_{\mathrm{S}}[\|\theta\|^{4}]}\mathrm{S}^{1/2}(\Theta^{\mathrm{c}}))
	\right]
	\nonumber\\
	&\geq
\left[\int 1_{\theta\in\Theta} R(\theta,\hat{\theta}_{\mathrm{S}})\mathrm{d}\mathrm{S}(\theta)
	-(2B\mathrm{S}(\Theta^{\mathrm{c}})+2cn^{-1}\mathrm{S}^{1/2}(\Theta^{\mathrm{c}}))
	\right],
	\end{align*}
	where $c:=\sum_{d=1}^{\infty}M(d)d^{2\alpha+1}<\infty$.
	To complete the proof,
	we use the following lemma.
	
	\begin{lem}
	\label{PriorMeasureonEllipsoid}
	There exists a constant $c_{1}$ depending on $a>0$ such that
	for a sufficiently large $n\in\mathbb{N}$, the inequality
	\begin{align*}
	\mathrm{S}\left( (\Theta)^{\mathrm{c}}\right)
	\leq c_{1}\mathrm{e}^{-a \lfloor (nB)^{1/(2\times(2\alpha+1))}\rfloor }
	\end{align*}
	holds.
	\end{lem}
	The proof of Lemma \ref{PriorMeasureonEllipsoid} is given after completing the proof of Theorem \ref{thm:Sobolevweakadmissibility_S}.

	By Lemma \ref{PriorMeasureonEllipsoid},
	for a sufficiently large $n>0$,
	we have
	\begin{align*}
		\inf_{\delta}\int R(\theta,\delta)\mathrm{d}\widetilde{\mathrm{S}}(\theta)
	\geq
	\left[\int R(\theta,\hat{\theta}_{\mathrm{S}})\mathrm{d}\widetilde{\mathrm{S}}(\theta)
		-\mathrm{O}\left(\exp\left\{-\frac{a}{2}T\right\}\right)
	\right],
	\end{align*}
	which demonstrates that inequality (\ref{eq:gammaBayes}) holds.
\end{proof}

\begin{proof}[Proof of Lemma \ref{PriorMeasureonEllipsoid}]
Let $T=\lfloor (nB)^{1/(4\alpha+2)} \rfloor$.
By definition,
\begin{align*}
	\mathrm{S}\left( \Theta^{\mathrm{c}}\right)
	&=
	\sum_{d=1}^{\infty}M(d)\mathrm{Pr}\left(\sum_{i=1}^{d}i^{-1}|N_{i}|^{2} > \frac{nB}{d^{2\alpha+1}}\right)
	\nonumber\\
	&\leq
	\sum_{d=1}^{T}M(d)\mathrm{Pr}\left(\sum_{i=1}^{d}i^{-1}|N_{i}|^{2} > \frac{nB}{d^{2\alpha+1}}\right)
	+\sum_{d=T+1}^{\infty}M(d)
	\nonumber\\
	&=
	\sum_{d=1}^{T}M(d)\mathrm{Pr}\left(\sum_{i=1}^{d}|N_{i}|^{2} > \frac{nB}{T^{2\alpha+1}}\right)
	+\sum_{d=T+1}^{\infty}M(d),
\end{align*}
where $\{N_{i}\}_{i=1}^{T}$ are independent random variables distributed according to $\mathcal{N}(0,1)$.

We next apply an exponential inequality for the chi-square statistics (Lemma 1 in \citet{LaurentandMassart(2000)}):
for any $x>d$,
\begin{align*}
	\mathrm{Pr}\left(\sum_{i=1}^{d}|N_{i}|^{2}\geq x^{2} \right)\leq \mathrm{e}^{-\frac{(\sqrt{x}-\sqrt{d})^{2}}{2}}.
\end{align*}
Setting $x=(1/2)(nB)^{1/2}$,
we have
\begin{align*}
	\mathrm{S}\left( \Theta^{\mathrm{c}} \right)
	\leq
	\sum_{d=1}^{T}M(d)\mathrm{e}^{-\frac{1}{8}(nB)^{1/2}(1+\mathrm{o}(nB)) }
	+\sum_{d=T+1}^{\infty}M(d).
\end{align*}
Setting $nB$ such that the $\mathrm{o}(nB)$ term in the above inequality is less than $1/2$ 
completes the proof.
\end{proof}

\section{Discussion and conclusions}\label{Section:Discussions}

In this paper,
we have demonstrated the usefulness of $\varepsilon$-admissibility in high-dimensional and nonparametric statistical models
by presenting two new results.
These results suggest the use of $\varepsilon$-admissibility in conjunction with the other criteria such as the minimax rate of convergence.

\appendix
\section{Poisson sequence model and inhomogeneous Poisson point process model}\label{Appendix:Poissonpointprocess}

In this appendix,
we discuss the relationship between Poisson sequence models and inhomogeneous Poisson point process models.
Let $\{N_{t}:t\in[0,1]\}$ be an inhomogeneous Poisson point process with an intensity function $\lambda(\cdot):[0,1]\to\mathbb{R}_{+}$.
We assume that the $L^{1}$-norm of $\lambda$ is bounded by $d$.
Let $h:=\lfloor1/d\rfloor$ be a time resolution.
From the independent incremental property of an inhomogeneous Poisson point process,
$d$ random variables $N_{h},N_{2h}-N_{h},\cdots$, $N_{1}-N_{h\times d}$ 
are independently distributed according to Poisson distributions
with means $\int_{0}^{h}\lambda(t)\mathrm{d}t$, $\int_{h}^{2h}\lambda(t)\mathrm{d}t$,$\ldots$,$\int_{h\times d}^{1}\lambda(t)\mathrm{d}t$,
respectively.
Thus 
letting $X_{1}:=N_{h}$,$X_{2}:=N_{2h}-N_{h}$,$\cdots$, $X_{d}:=N_{1}-N_{h\times d}$
and
letting $\theta_{1}=\int_{0}^{h}\lambda(t)\mathrm{d}t$, $\theta_{2}=\int_{h}^{2h}\lambda(t)\mathrm{d}t$,
$\ldots$,
$\theta_{d}=\int_{h\times d}^{1}\lambda(t)\mathrm{d}t$
yields the Poisson sequence model presented in Section \ref{Section:PoissonL1}.

\section{Gaussian sequence model with an $\mathcal{L}^{2}$-constraint parameter space}
\label{Appendix:Additionalexample}

In this appendix,
we consider estimation in a $d$-dimensional Gaussian sequence model with $\mathcal{L}^{2}$-constraint parameter space.
Let $\mathcal{X}=\mathbb{R}^{d}$, $\Theta=\{ \theta=(\theta_{1}, \ldots, \theta_{d}): \sum_{i=1}^{d}\theta_{i}^{2} / d  \leq 1 \}$,
$\mathcal{P} = \{P_{\theta}=\otimes_{i=1}^{d} \mathcal{N}(\theta_{i} , 1) \}$,
$\mathcal{A} = \mathbb{R}^{d}$,
and
$L(\theta , a ) = \|\theta - a\|^{2}$,
where $\|\cdot\|$ is the $l_{2}$-norm in $\mathbb{R}^{d}$.
We compare the following two estimators:
one is the maximum likelihood estimator $\hat{\theta}_{\mathrm{MLE}} (X) := X$;
the other is the James--Stein estimator $\hat{\theta}_{\mathrm{JS}} (X) := (1- (d-1) / \|X\|^{2}) X$.

First,
the following lemma shows that the rate of convergence of $\hat{\theta}_{\mathrm{MLE}}$ is equal to that of a minimax risk,
which indicates that from the viewpoint of the rate of convergence we can not determine whether $\hat{\theta}_{\mathrm{JS}}$ is superior to $\hat{\theta}_{\mathrm{MLE}}$.
\begin{lem}[Theorems 7.28 and 7.48 in \citet{Wasserman_Book}]
	\label{minimaxrate_GS_L2}
	We have
	\begin{align}
		\mathop{\lim}_{d\to\infty}\frac{\inf_{\hat{\theta}\in\mathcal{D}}\sup_{\theta\in\Theta}R(\theta,\hat{\theta})}{d}
		=\frac{1}{2}.
		\label{eq_minimax_GS_L2}
	\end{align}
	We also have
	\begin{align}
		\sup_{\theta\in\Theta}R(\theta,\hat{\theta}_{\mathrm{MLE}})
		\asymp
		\sup_{\theta\in\Theta}R(\theta,\hat{\theta}_{\mathrm{JS}})
		\asymp d .
	\end{align}
\end{lem}

Next,
the following theorem tells that from the viewpoint of weak admissibility,
$\hat{\theta}_{\mathrm{JS}}$ is better than $\hat{\theta}_{\mathrm{MLE}}$ as $d$ grows to infinity.
\begin{thm}\label{thm:Gaussianweakadmissibility}
	We have
	\begin{align*}
	 	\mathcal{R}(\Theta,\hat{\theta}_{\mathrm{JS}}) \lesssim 1 \lesssim d \lesssim \mathcal{R}(\Theta,\hat{\theta}_{\mathrm{MLE}}).
	\end{align*}
\end{thm}

Remark that,
unlike the examples in Sections \ref{Section:PoissonL1} and \ref{Section:Gaussiansequence},
we also use multiplicative constant terms 
$\lim \sup_{\theta\in\Theta} R(\theta,\hat{\theta}_{\mathrm{JS}}) / d$
and 
$\lim \sup_{\theta\in\Theta} R(\theta,\hat{\theta}_{\mathrm{MLE}}) /d$
to compare these estimators:
the former is 1/2, whereas the latter is 1.

\vspace{5mm}
The proof of this theorem is a simple combination of the following lemmas.
\begin{lem}
	For any $d(\in\mathbb{N})>2$, we have
\begin{align*}
	\mathcal{R}(\Theta,\hat{\theta}_{\mathrm{MLE}})\geq \frac{(d-2)^{2}}{2d}.
\end{align*}
\end{lem}
\begin{proof}
	The proof follows the line of the well-known proof that the James--Stein estimator dominates the maximum likelihood estimator.
	Applying Lemma \ref{lower_maximin} with $\tilde{\delta}=\hat{\theta}_{\mathrm{JS}}$
	yields
	\begin{align}
		\mathcal{R}(\Theta,\hat{\theta}_{\mathrm{MLE}})
		&\geq \inf_{\theta\in\Theta}[R(\theta,\hat{\theta}_{\mathrm{MLE}})-R(\theta,\hat{\theta}_{\mathrm{JS}})]
		\nonumber\\
		&=\inf_{\theta\in\Theta}\mathrm{E}_{X|\theta}
		\left[\|\theta-X\|^{2}-\|\theta-X+(d-2)X/\|X\|\|^{2}\right]
		\nonumber\\
		&=\inf_{\theta\in\Theta}\mathrm{E}_{X|\theta}
		\left[-\frac{(d-2)^{2}}{\|X\|^{2}}+2\left\langle X-\theta,\frac{(d-2)X}{\|X\|^{2}} \right\rangle \right]
		\nonumber\\
		&=\inf_{\theta\in\Theta}\mathrm{E}_{X|\theta}
		\left[-\frac{(d-2)^{2}}{\|X\|^{2}}+2(d-2) \sum_{i=1}^{d}\frac{\|X\|^{2}-2X^{2}_{i}}{\|X\|^{4}}\right]
		\nonumber\\
		&=(d-2)^{2} \inf_{\theta\in\Theta}\mathrm{E}_{X|\theta}
		\left[\frac{1}{\|X\|^{2}}\right]
		\nonumber\\
		&\geq \inf_{\theta\in\Theta}\left[\frac{(d-2)^{2}}{d+\|\theta\|^{2}}\right].
		\label{Riskdiff_MLE_JS_GS_L2}
	\end{align}
	Here, the third equality follows from Stein's lemma
	and from that $\partial(X/\|X\|^{2}) / \partial X_{i} = (\|X\|^{2}-2X_{i}) / \|X\|^{4}$.
	The last inequality follows from Jensen's inequality and from that $\mathrm{E}_{X|\theta}\|X\|^{2}=d+\|\theta\|^{2}$.
	Thus, we complete the proof.
\end{proof}

\begin{lem}
	For any $d(\in\mathbb{N})>2$, we have
\begin{align*}
	\mathcal{R}(\Theta,\hat{\theta}_{\mathrm{JS}})\leq 2-\frac{4}{d}.
\end{align*}
\end{lem}

\begin{proof}
	Applying Lemma \ref{upper_maximin} with $\Pi:=\delta_{0}$ yields
	\begin{align*}
		\mathcal{R}(\Theta,\hat{\theta}_{\mathrm{JS}})
		\leq 
		R(0,\hat{\theta}_{\mathrm{JS}}).
	\end{align*}
	From (\ref{Riskdiff_MLE_JS_GS_L2}), we have $R(0,\hat{\theta}_{\mathrm{JS}})\leq (2d-4)/d$,
	which completes the proof.
\end{proof}

\section{Calculation used in Theorem \ref{thm:Sobolevweakadmissibility_G}}\label{Appendix:details_firstterm}
In this appendix,
we provide a calculation that is used in Theorem \ref{thm:Sobolevweakadmissibility_G}.
Let
\begin{align*}
	f(x,y):=\frac{y^{2}x^{2\alpha+2}}{(1+yx^{2\alpha+1})^{2}}.
\end{align*}
Then,
the risk difference of $R(\theta,\hat{\theta}_{\mathrm{G}})-R(\theta,\hat{\theta}_{\mathrm{G}(s)})$
is
\begin{align}
R(\theta,\hat{\theta}_{\mathrm{G}})-&R(\theta,\hat{\theta}_{\mathrm{G}(s)})
	\nonumber\\
	=&\frac{\sum_{j=1}j^{2\alpha}\theta_{j}^{2}}{n^{-1}} \sum_{i=1}^{\infty}
	\frac{i^{2\alpha} \theta_{i}^{2}}{\sum_{j=1}^{\infty}j^{2\alpha}\theta^{2}_{j}}
	\left\{f(i,1/n)-f(i,1/\{ns\})\right\}
	\nonumber\\
	&+\sum_{i=1}^{\infty}\left\{\left(\frac{1}{1+i^{2\alpha+1}/n}\right)^{2}-\left(\frac{1}{1+i^{2\alpha+1}/\{n\widetilde{B}\}}\right)^{2} \right\}.
	\label{Riskdifference_1}
\end{align}
We will show that the following inequality holds:
\begin{align*}
R&(\theta,\hat{\theta}_{\mathrm{G}})-R(\theta,\hat{\theta}_{\mathrm{G}(s)})
\nonumber\\
\geq&B\left(\frac{2\alpha+2}{4\alpha+1}\right)^{\frac{2\alpha+2}{2\alpha+1}}n^{-\frac{2\alpha}{2\alpha+1}}
	\left[
		\frac{1}{\{1+\frac{2\alpha+2}{4\alpha+1}\}^{2}}-\frac{s^{-2}}{\{1+\frac{2\alpha+2}{s(4\alpha+1)}\}^{2}}
	\right]
	\nonumber\\
	&+n^{-1}\sum_{i=1}^{\infty}\left\{\left(\frac{1}{1+i^{2\alpha+1}/n}\right)^{2}-\left(\frac{1}{1+i^{2\alpha+1}/\{ns\}}\right)^{2}\right\}.
\end{align*}

To derive this,
we show that the lower bound of $f(i,1/n)-f(i,1/\{nB\})$ for all $i\in\mathbb{N}$ is given by
\begin{align}
	f(i,1/n)-f(i,1/\{n\widetilde{B}\})\geq f(x^{*},1/n)-f(x^{*},1/\{ns\}),
	\label{lowerbound_f}
\end{align}
where
\begin{align}
	x^{*}:=\left(\{2\alpha+2\}/\{4\alpha+1\}\right)^{\frac{1}{2\alpha+1}}n^{\frac{1}{2\alpha+1}}.
	\label{def_x}
\end{align}
For a fixed $y$,
\begin{align*}
	\partial f(x,y) / \partial x
	=[y^{2} x^{2\alpha+1}\{(2\alpha+2)-2\alpha y x^{2\alpha+1}\}]/(1+yx^{2\alpha+1})^{3}.
\end{align*}
Letting 
\begin{align*}
	g(y;x) := \{(2\alpha+2)x^{2\alpha+1}y^2-(2\alpha)x^{4\alpha+2}y^{3}\} / (1+yx^{2\alpha+1})^{3}
\end{align*}
yields
\begin{align*}
\partial f(x,y)/\partial x- \partial f(x,y/\sqrt{s}) / \partial x = g(y;x)-g(y/\sqrt{s};x).
\end{align*}
Since
\begin{align*}
	\frac{\partial g(y;x)}{\partial y}=\frac{x^{2\alpha+1}y}{(1+x^{2\alpha+1}y)^{4}}\left\{(4\alpha+4)-(8\alpha+2)x^{2\alpha+1}y\right\},
\end{align*}
we have
\begin{align*}
 \partial f(x,y) / \partial x - \partial f(x,y/\sqrt{s}) / \partial x \geq 0 \text{ if $x^{2\alpha+1}y\geq (4\alpha+4)/(8\alpha+2)$}
\end{align*}
and
\begin{align*}
 \partial f(x,y) / \partial x - \partial f(x,y/\sqrt{s}) / \partial x \leq 0 \text{ if $x^{2\alpha+1}y\leq (4\alpha+4) / (8\alpha+2) $}.
\end{align*}
Therefore,
\begin{align*}
	f(i,1/n)-f(i,1/\{ns\})\geq f(x^{*},1/n)-f(x^{*},1/\{ns\}).
\end{align*}

From the inequality that $\sum_{i=1}^{\infty}p(i)f(i)\geq L$ 
for a function $f:\mathbb{N}\to [L,\infty)$ with some $L\in\mathbb{R}$
and for a probability $p$ on $\mathbb{N}$,
and
from the inequality (\ref{lowerbound_f}),
regarding $(1^{2\alpha}\theta_{1}^{2},2^{2\alpha}\theta_{2}^{2},\ldots)/\sum_{j=1}^{\infty}j^{2\alpha}\theta_{j}^{2}$ as a probability
completes the proof.

\bibliographystyle{imsart-nameyear}
\bibliography{Weakadmissible}
\end{document}